\documentclass[11pt]{article}
\usepackage{amsfonts}
\usepackage{amsmath}
\newcommand{\func}[1]{\operatorname{#1}}

\newtheorem{theorem}{Theorem}

\newtheorem{corollary}[theorem]{Corollary}

\newtheorem{definition}[theorem]{Definition}

\newtheorem{remark}[theorem]{Remark}

\newenvironment{proof}[1][Proof]{\noindent\textbf{#1.} }{\ \rule{0.3em}{0.3em}}
\numberwithin{equation}{section}

\begin{document}

\title{Stochastic applications of Caputo-type convolution operators with
non-singular kernels }
\author{Luisa Beghin \thanks{%
Sapienza University of Rome, Italy} \and Michele Caputo \thanks{%
Accademia Nazionale dei Lincei, Italy} }
\date{}
\maketitle

\begin{abstract}
We consider here convolution operators, in the Caputo sense, with
non-singular kernels. We prove that the solutions to some
integro-differential equations with such operators (acting on the space
variable) coincide with the transition densities of a particular class of L%
\'{e}vy subordinators (i.e. compound Poisson processes with non-negative
jumps). We then extend these results to the case where the kernels of the
operators have random parameters, with given distribution. This assumption
allows greater flexibility in the choice of the kernel's parameters and,
consequently, of the jumps' density function.


\end{abstract}

\section{Introduction}

The convolution operators with non-singular kernels have drawn, in recent
years, a wide interest, from both the theoretical and the applied points of
view: see, for example, \cite{LOS}, \cite{ATAN}, \cite{YAN}, \cite{ZHE},
\cite{ARS} and the references therein.

Let $f:\mathbb{R}^{+}\rightarrow \mathbb{R}$ be a differentiable and
absolutely integrable function in $AC_{loc}\left( 0,+\infty \right) $ (i.e.
with integrable first derivative); we use here the following operator
\begin{equation}
^{CF}D_{x}^{\alpha }f(x):=\frac{B(\alpha )}{1-\alpha }\int\limits_{0}^{x}%
\frac{d}{dz}f(z)e^{-\frac{\alpha }{1-\alpha }(x-z)}dz,\qquad x>0,\text{ }%
\alpha \in (0,1),  \label{cf}
\end{equation}%
(where $B(\alpha )$ is a positive normalizing constant), which is a variant
of the so-called \textquotedblleft Caputo-Fabrizio fractional derivative"
introduced in \cite{CAP}: in our case the lower limit of integration is
zero, since we identify the kernel with the tail of a (bounded) L\'{e}vy
measure on $(0,+\infty ).$

Analogously, we define the following convolution operator

\begin{equation}
D_{x}^{\alpha ,\nu }f(x):=\frac{B(\alpha )}{1-\alpha }\int\limits_{0}^{x}%
\frac{d}{dz}f(z)E_{\nu }\left( -\frac{\alpha }{1-\alpha }(x-z)^{\nu }\right)
dz,\qquad x>0,\text{ }\alpha \in (0,1),\nu \in (0,1],  \label{ab}
\end{equation}%
where $E_{\nu }(\cdot )$ is the Mittag-Leffler function, i.e. $E_{\nu
}(x):=\sum_{j=0}^{\infty }x^{j}/\Gamma (\nu j+1),$ for $x\in \mathbb{R},$ $%
\func{Re}(\nu )>0$. We note that (\ref{ab}) reduces, for $\alpha =\nu $, to
the so-called \textquotedblleft Atangana-Baleanu fractional derivative" (in
the Caputo sense), see \cite{ATA}, while, for $\nu =1$, it coincides with (%
\ref{cf}). As for (\ref{cf}), the kernel of (\ref{ab}) is non-singular in
the origin, since $E_{\nu }(0)=1.$

Moreover, we introduce here the following operator%
\begin{equation}
\mathcal{D}_{x}^{\alpha ,\rho }f(x):=\frac{1}{\Gamma (\rho )}%
\int\limits_{0}^{x}\frac{d}{dz}f(z)\Gamma \left( \rho ;k_{\alpha }z\right)
dz,\qquad x>0,\text{ }k_{\alpha }>0,\rho \in (0,1],  \label{ig}
\end{equation}%
where $\Gamma (\rho ;x):=\int_{x}^{+\infty }e^{-w}w^{\rho -1}dw$ is the
upper-incomplete gamma function. Also (\ref{ig}) generalizes (\ref{cf}), to
which it reduces for $\rho =1$ and $k_{\alpha }=\alpha /(1-\alpha ).$ Again
the kernel of (\ref{ig}) is non-singular in the origin, since $\Gamma \left(
\rho ;0\right) =\Gamma \left( \rho \right) .$

Our aim is to analyze the role of this kind of operators in the field of
stochastic processes; while the random models associated to differential
equations with classical fractional derivatives have been extensively
studied, the probabilistic applications of the above defined operators are
not yet explored.

Let us start by recalling that $\psi $:$(0,+\infty )\rightarrow \mathbb{R}%
^{+}$ is a Bernstein function if it is non-negative, of class $C^{\infty }$
and such that, for any $x>0,$ $(-1)^{k}\frac{d^{k}}{dx^{k}}\psi (x)\leq 0,$ $%
k\in \mathbb{N}.$ Then it is well-known that $\psi $ admits the following
representation (see \cite{SCH}, p.21)%
\begin{equation}
\psi (x)=c+bx+\int\limits_{0}^{+\infty }(1-e^{-zx})\mu (dz),  \label{bb}
\end{equation}%
where $c,b\geq 0$ and $\mu $ is a measure on $(0,+\infty )$ satisfying $%
\int\limits_{0}^{+\infty }(1\wedge z)\mu (dz)<\infty $, called L\'{e}vy
measure$.$ Moreover, the triplet $\left( c,b,\mu \right) $ determines
uniquely $\psi $ (and the reverse holds as well).

Let us define the stochastic process $\mathcal{S}:=\mathcal{S}(t),t\geq 0$,
assuming that it is a subordinator, i.e. L\'{e}vy and almost surely
non-decreasing. Let $h_{\mathcal{S}}(B,t):=P(\mathcal{S}(t)\in B)$ be its
transition probabilities, for any $t\geq 0$ and Borel interval $B\subset
(0,+\infty ).$ Let now $\overline{\mu }(s):=\int_{s}^{+\infty }\mu (dz),$
for $s\geq 0,$ be the so-called tail L\'{e}vy measure; if $\overline{\mu }%
(\cdot )$ is absolutely continuous on $(0,+\infty )$ and $\int_{0}^{+\infty
}\mu (z)dz=+\infty $, the corresponding subordinator $\mathcal{S}$ has
absolutely continuous distribution (see \cite{SAT}, p.177), with density $h_{%
\mathcal{S}}(dx,t)=h_{\mathcal{S}}(x,t)dx,$ $t,x\geq 0,$ such that%
\begin{equation}
\mathbb{E}e^{-\eta \mathcal{S}(t)}:=\int\limits_{0}^{+\infty }e^{-\eta x}h_{%
\mathcal{S}}(x,t)dx=e^{-\psi (\eta )t},\qquad \eta >0,\;t\geq 0  \label{lp2}
\end{equation}%
(see \cite{SCH}, p.49). We define now the convolution operator $\mathcal{D}%
_{x}^{\psi }$ as follows%
\begin{equation}
\mathcal{D}_{x}^{\psi }f(x):=\int\limits_{0}^{x}\frac{d}{dz}f(z)\overline{%
\mu }_{\psi }(x-z)dz,  \label{co}
\end{equation}%
for any differentiable and absolutely integrable function in $AC_{loc}\left(
0,+\infty \right) .$ If $f(0)<\infty ,$ it can be alternatively written as $%
\mathcal{D}_{x}^{\psi }f(x)=\frac{d}{dz}\int\limits_{0}^{x}f(z)\overline{\mu
}_{\psi }(x-z)dz-\overline{\mu }_{\psi }(x)f(0)$ (see \cite{KOC} and \cite%
{KOC2}). Note that we have emphasized the dependence of the tail L\'{e}vy
measure on the Bernstein function by denoting it as $\overline{\mu }_{\psi
}(\cdot )$; analogously, we will denote by $\mathcal{S}_{\psi }$ the
subordinator with L\'{e}vy triplet $(0,0,\mu _{\psi }),$ under the
assumption that $c=b=0,$ without loss of generality$.$ The Laplace transform
of (\ref{co}) can be easily obtained, if again $f(0)<\infty $, by
considering that $\int\limits_{0}^{+\infty }e^{-\eta x}\overline{\mu }_{\psi
}(x)dx=\psi (\eta )/\eta ,$ and reads%
\begin{equation}
\mathcal{L}\left\{ \mathcal{D}_{x}^{\psi }f(x);\eta \right\} =\psi (\eta )%
\widetilde{f}(\eta )-\frac{\psi (\eta )}{\eta }f(0),  \label{lt3}
\end{equation}%
where $\widetilde{f}(\eta ):=\mathcal{L}\left\{ f(x);\eta \right\} $ and $%
\mathcal{L}\left\{ \cdot ;\eta \right\} $ denotes the Laplace transform with
respect to $x$. It is proved in \cite{TOA} that, if the L\'{e}vy measure is
absolutely continuous and unbounded and if the subordinator $\mathcal{S}%
_{\psi }$ has density $h_{\mathcal{S}_{\psi }}(x,t)$ vanishing in the
origin, the following initial value problem%
\begin{equation}
\left\{
\begin{array}{l}
\frac{\partial }{\partial t}h(x,t)=-\mathcal{D}_{x}^{\psi }h(x,t) \\
h(0,t)=0,\qquad h(x,0)=\delta (x)%
\end{array}%
\right. ,  \label{cau}
\end{equation}%
is satisfied by $h_{\mathcal{S}_{\psi }}(x,t)$ for $x,t\geq 0.$

Since we consider here integral operators with non-singular kernel in the
origin, such as (\ref{cf}), (\ref{ab}) and (\ref{ig}), we move to the case
where the L\'{e}vy measure has finite mass, i.e. $\int_{0}^{+\infty }\mu
_{\psi }(dz)<\infty ,$ assuming also that $b=0$ in (\ref{bb}), without loss
of generality$.$ The corresponding Bernstein function $\psi $ is
consequently bounded and the subordinator $\mathcal{S}_{\psi }$ is a
driftless, step process. Moreover, almost all its paths have a finite number
of jumps on every compact interval (finite activity). We will make these
assumptions, together with $c=0,$ which implies that $\mathcal{S}_{\psi }$
is a strict subordinator. Then, in our case, the condition for the absolute
continuity of $h_{\mathcal{S}}(\cdot ,t)$ does no longer hold. Moreover, in
some cases considered here also the assumption of density vanishing in the
origin does not hold; on the contrary, as happens even for well-known
processes (such as, for example, the gamma subordinators), the density is
infinite for $x=0.$ Thus the result in (\ref{cau}) must be modified
accordingly.

It is well known that, under the assumption that $\int_{0}^{+\infty }\mu
_{\psi }(dz)<\infty ,$ a L\'{e}vy process with triplet $(0,0,\mu _{\psi })$
is a compound Poisson process, i.e.
\begin{equation}
\mathcal{S}_{\psi }(t)=\sum_{j=1}^{N(t)}X_{j}^{\psi },  \label{ss}
\end{equation}%
where $N:=N(t),t\geq 0$ is a Poisson process with parameter $\lambda =1$,
independent of $X_{j}^{\psi },$ for any $j=1,2...$ (see, for example, \cite%
{APP}, p.49). Moreover, the addends $X_{j}^{\psi }$ are, for $j=1,2,...,$
non-negative, independent, identically distributed (i.i.d.) random variables
with $F_{X_{j}^{\psi }}(x):=P(X_{j}^{\psi }\leq x)$ such that%
\begin{equation}
\mathcal{L}\left\{ F_{X_{j}^{\psi }}(y);\eta \right\} =\frac{1}{\eta }\left[
1-\psi (\eta )\right] .  \label{fr}
\end{equation}%
In this case, the distribution function of $\mathcal{S}_{\psi }$ reads, for $%
y\in \mathbb{R},$%
\begin{equation}
F_{\mathcal{S}_{\psi }}(y,t):=P\{\mathcal{S}_{\psi
}(t)<y\}=e^{-t}1_{[0,+\infty )}(y)+\int_{-\infty }^{y}f_{\mathcal{S}_{\psi
}}(x,t)dx,  \label{cc}
\end{equation}%
where
\begin{equation}
f_{\mathcal{S}_{\psi }}(x,t)=e^{-t}\sum_{n=1}^{\infty }\frac{t^{n}}{n!}%
f_{X_{j}^{\psi }}^{\ast (n)}(x),  \label{lp}
\end{equation}%
is the density of the absolutely continuous component and $f^{\ast (n)}$
denotes the $n$-fold convolution of the function $f$. The compound Poisson
process has important applications in different fields, ranging from models
of insurance risk, to the analysis of statistical behavior in social and
biological systems, as well as to the treatment of certain types of random
dynamics in physics.

As a consequence of (\ref{cc}) and (\ref{lp}), the Laplace transform of $f_{%
\mathcal{S}_{\psi }}$is given by%
\begin{eqnarray}
\widetilde{f}_{\mathcal{S}_{\psi }}(\eta ,t) &=&e^{-t}\sum_{n=1}^{\infty }%
\frac{1}{n!}\left( \widetilde{f}_{X^{\psi }}(\eta )t\right) ^{n}
\label{lap3} \\
&=&[\text{by (\ref{fr})}]  \notag \\
&=&e^{-t}\sum_{n=1}^{\infty }\frac{1}{n!}\left( (1-\psi (\eta )t\right)
^{n}=e^{-\psi (\eta )t}-e^{-t},\qquad \eta >0,\;t\geq 0,  \notag
\end{eqnarray}%
instead of (\ref{lp2}). Correspondingly, as we will prove in the next
section, the equation satisfied by the density $f_{\mathcal{S}_{\psi }}$
differs from (\ref{cau}) by two additional terms, which depend on the choice
of $\psi $ and whether or not the density of the subordinator is infinite in
the origin, for some values of $t$.

In the last section, we extend these results by generalizing the previous
operators to the case of random parameters, thus obtaining distributed-order
convolution operators. We provide the explicit solution of the corresponding
equations, at least under simplifying assumptions.

Finally, in the concluding remarks, we hint some applications of the
obtained results to the risk theory and, in particular, to a continuous-time
model, where the surplus process of the insurance company is modelled by a
compound Poisson process with non-negative, absolutely continuous claim
sizes.

\

We recall the following definitions of well-known special functions that we
will apply later: let $W_{\alpha ,\beta }(x):=\sum_{j=0}^{\infty
}x^{j}/j!\Gamma (\alpha j+\beta ),$ for $x,\alpha ,\beta \in \mathbb{C},$ be
the Wright function and let

\begin{equation}
E_{\alpha ,\beta }^{\gamma }(x):=\sum_{j=0}^{\infty }\frac{x^{j}(\gamma )_{j}%
}{j!\Gamma (\alpha j+\beta )},\qquad \func{Re}(\alpha ),\func{Re}(\beta
)>0,\;\gamma >0,\;x\in \mathbb{R},  \label{ml}
\end{equation}%
where $(\gamma )_{j}:=\gamma (\gamma +1)...(\gamma +j-1),$ $j=0,1,...,$ be
the Prabhakar function (or Mittag-Leffler function with three parameters).
We will denote, for brevity, the function (\ref{ml}), as $E_{\alpha ,\beta
}(x)$ when $\gamma =1$, and as $E_{\alpha }(x)$, when $\gamma =\beta =1.$
Let us recall the following formula for the Laplace transform of (\ref{ml}):%
\begin{equation}
\mathcal{L}\left\{ x^{\beta -1}E_{\alpha ,\beta }^{\gamma }(\lambda
x^{\alpha });\eta \right\} =\frac{\eta ^{\alpha \gamma -\beta }}{(\eta
^{\alpha }-\lambda )^{\gamma }},  \label{ml2}
\end{equation}%
for $\func{Re}(\eta )>0,$ $\func{Re}(\beta )>0,$ $\lambda \in \mathbb{C}$
and $|\lambda \eta ^{-\alpha }|<1$ (see \cite{KIL}, p.47).

\section{Main results}

Let us consider the convolution operator $\mathcal{D}_{x}^{\psi }$ defined
in (\ref{co}) under different assumptions on the kernel $\overline{\mu }%
_{\psi }(\cdot )$: in the exponential case, $\mathcal{D}_{x}^{\psi }$
reduces to the variant of the Caputo-Fabrizio operator defined in (\ref{cf})
and the solution is finite in the origin. Then we analyze two cases, where
the kernel of the operator is represented by a Mittag-Leffler (with
parameter $\nu \in (0,1]$) or an incomplete gamma function (with parameter $%
\rho \in (0,1]$). These extensions are both quite natural, since the
Mittag-Leffler density is the fractional counterpart of the exponential one;
on the other hand, the incomplete gamma function coincides with the tail
distribution function of a gamma random variable, which generalizes the
exponential. As we will see, even though the last two operators reduce to
the first one, for $\nu =1$ and $\rho =1$, respectively, the exponential
case must be treated separately, since the governing equation is not
accordingly obtained as special case. This is a consequence of the different
behavior of the solutions in the origin.

\subsection{Exponential kernel}

\begin{theorem}
Let $^{CF}D_{x}^{\alpha }$ be the convolution operator defined in (\ref{cf}%
), with $B(\alpha )=1-\alpha $. Then the solution to the following equation%
\begin{equation}
\frac{\partial }{\partial t}f(x,t)=-\,^{CF}D_{x}^{\alpha }\,f(x,t)+k_{\alpha
}(1-t)e^{-t-k_{\alpha }x},\qquad x\geq 0,t\geq 0,\text{ }\alpha \in (0,1),
\label{eq}
\end{equation}%
with $k_{\alpha }=\alpha /(1-\alpha )$ and initial condition $f(x,0)=0,$ is
given by
\begin{equation}
f_{\mathcal{S}_{\psi }}(x,t)=k_{\alpha }t\exp \left\{ -k_{\alpha
}x-t\right\} W_{1,2}\left( k_{\alpha }xt\right) ,\qquad x\geq 0,  \label{den}
\end{equation}%
and (\ref{den}) is the density of the absolutely continuous component of $%
\mathcal{S}_{\psi }$ defined in (\ref{ss}), for $X_{j}^{\psi }$
exponentially distributed with parameter $k_{\alpha }$, for $j=1,2...$.
\end{theorem}

\begin{proof}
We first prove that the density $f_{\mathcal{S}_{\psi }}$ of the absolutely
continuous component of (\ref{ss}) satisfies the following equation%
\begin{equation}
\frac{\partial }{\partial t}f(x,t)=-\mathcal{D}_{x}^{\psi }f(x,t)-\overline{%
\mu }_{\psi }(x)f(0,t)+f_{X_{j}^{\psi }}(x)e^{-t},\qquad x,t\geq 0.
\label{cau3}
\end{equation}%
Indeed, by taking the time-derivative of (\ref{lap3}), we get%
\begin{equation*}
\frac{\partial }{\partial t}\widetilde{f}_{\mathcal{S}_{\psi }}(\eta
,t)=-\psi (\eta )e^{-\psi (\eta )t}+e^{-t}.
\end{equation*}%
which coincides with the Laplace transform of the r.h.s. of (\ref{cau3}), by
considering (\ref{lt3}) and that $\int\limits_{0}^{+\infty }e^{-\eta x}%
\overline{\mu }_{\psi }(x)dx=\psi (\eta )/\eta $. Equation (\ref{cau3})
obviously holds only when $f(0,t)<\infty ,$ for any $t.$ By definition (\ref%
{cf}) we can write, in this case, $\overline{\mu }(x)=e^{-k_{\alpha }x}$ and
\begin{equation*}
\,\psi (\eta )=\eta \int\limits_{0}^{+\infty }e^{-\eta x-k_{\alpha }x}dx=%
\frac{\eta }{\eta +k_{\alpha }}.
\end{equation*}%
Then equation (\ref{cau3}) coincides with (\ref{eq}) and (\ref{lap3})
reduces to%
\begin{equation}
\widetilde{f}_{\mathcal{S}_{\psi }}(\eta ,t)=e^{-\frac{\eta t}{\eta
+k_{\alpha }}}-e^{-t},  \label{hh}
\end{equation}%
whose inverse is equal to (\ref{den}) and is finite also for $x=0$.
Moreover, by considering (\ref{fr}), we have that $f_{X_{j}^{\psi }}(x)=%
\mathcal{L}^{-1}\{k_{\alpha }/(\eta +k_{\alpha });x\}=k_{\alpha
}e^{-k_{\alpha }x}$, where $\mathcal{L}^{-1}\{\cdot ;x\}$ denotes the
inverse Laplace transform.
\end{proof}

\begin{remark}
This result can be alternatively checked directly, by differentiating (\ref%
{den}) and applying definition (\ref{cf}): indeed we have that%
\begin{eqnarray}
&&^{CF}D_{x}^{\alpha }\,f_{\mathcal{S}_{\psi }}(x,t)  \label{uf2} \\
&=&k_{\alpha }te^{-t-k_{\alpha }x}\sum_{n=1}^{\infty }\frac{(k_{\alpha
}t)^{n}}{(n-1)!(n+1)!}\int_{0}^{x}z^{n-1}dz-k_{\alpha }^{2}te^{-t-k_{\alpha
}x}\sum_{n=0}^{\infty }\frac{(k_{\alpha }t)^{n}}{n!(n+1)!}\int_{0}^{x}z^{n}dz
\notag \\
&=&k_{\alpha }te^{-t-k_{\alpha }x}\sum_{n=1}^{\infty }\frac{(k_{\alpha
}tx)^{n}}{n!(n+1)!}-k_{\alpha }e^{-t-k_{\alpha }x}\sum_{l=1}^{\infty }\frac{%
(k_{\alpha }tx)^{l}}{l!^{2}},  \notag
\end{eqnarray}%
which proves equation (\ref{eq}). Formula (\ref{den}) coincides with the
special case, for $\beta =1,$ of the distribution of the time-fractional
compound Poisson process (with exponential addends) obtained in \cite{BEG}.
Moreover, an alternative expression, in terms of modified Bessel functions,
is given in \cite{ROL}.
\end{remark}

\subsection{Mittag-Leffler kernel}

We now consider the operator defined in (\ref{ab}). Even if the latter
reduce to (\ref{cf}) for $\nu =1,$ the following result holds only for $\nu
<1,$ as explained in the Remark 4 below.

\begin{theorem}
Let $D_{x}^{\alpha ,\nu }$ be the convolution operator (\ref{ab}), with $%
B(\alpha )=1-\alpha $, for $\alpha ,\nu \in (0,1)$ and $k_{\alpha }=\alpha
/(1-\alpha )$. Then the solution of the following equation%
\begin{equation}
\frac{\partial }{\partial t}f(x,t)=-\,D_{x}^{\alpha ,\nu }\,f(x,t)+k_{\alpha
}e^{-t}x^{\nu -1}E_{\nu ,\nu }(-k_{\alpha }x^{\nu }),\qquad x\geq 0,t\geq 0,
\label{abeq}
\end{equation}%
with initial condition $f(x,0)=0$, coincides with the density function
\begin{equation}
f_{\mathcal{S}_{\psi }}(x,t)=\frac{e^{-t}}{x}\sum_{n=1}^{\infty }\frac{%
\left( k_{\alpha }tx^{\nu }\right) ^{n}}{n!}E_{\nu ,\nu n}^{n}(-k_{\alpha
}x^{\nu }).  \label{abden}
\end{equation}
\end{theorem}

\begin{proof}
In this case $\overline{\mu }(x)=E_{\nu }\left( -k_{\alpha }x^{\nu }\right) $
and, by considering (\ref{ml2}), we get%
\begin{equation*}
\,\psi (\eta )=\eta \int\limits_{0}^{+\infty }e^{-\eta x}E_{\nu }\left(
-k_{\alpha }x^{\nu }\right) dx=\frac{\eta ^{\nu }}{\eta ^{\nu }+k_{\alpha }}.
\end{equation*}%
Thus, we have that%
\begin{equation}
\widetilde{f}_{\mathcal{S}_{\psi }}(\eta ,t)=e^{-\frac{\eta ^{\nu }t}{\eta
^{\nu }+k_{\alpha }}}-e^{-t}  \label{ml3}
\end{equation}%
The inverse Laplace transform of (\ref{ml3}) gives (\ref{abden}), as can be
easily checked by (\ref{ml2}). It has been proved in \cite{BEG} that (\ref%
{abden}) coincides with the density of the absolutely continuous component
of the compound Poisson process $\mathcal{S}_{\psi }$, under the assumption
that the addends $X_{j}^{\psi }$ are i.i.d. random variables, for $%
j=1,2,..., $ with density function%
\begin{equation}
f_{X_{j}^{\psi }}(x)=k_{\alpha }x^{\nu -1}E_{\nu ,\nu }(-k_{\alpha }x^{\nu
}),\qquad x\geq 0.  \label{xx}
\end{equation}%
In this case we must take into account that, for $\nu <1$, the density (\ref%
{abden}) is infinite when $x=0$; thus we can not apply the Laplace transform
formula (\ref{lt3}); however, for a function $f$ such that the Laplace
transform of the derivative exists, we can write that%
\begin{equation}
\mathcal{L}\{D_{x}^{\alpha ,\nu }\,f(x);\eta \}=\mathcal{L}\{\frac{d}{dx}%
\,f(x);\eta \}\mathcal{L}\{E_{\nu }\left( -k_{\alpha }x^{\nu }\right) ;\eta
\},  \label{ll}
\end{equation}%
by considering (\ref{ab}). The space-derivative of (\ref{abden}) is equal to%
\begin{equation}
\frac{\partial }{\partial x}f_{\mathcal{S}_{\psi
}}(x,t)=e^{-t}\sum_{n=1}^{\infty }\frac{\left( k_{\alpha }t\right) ^{n}}{n!}%
\sum_{j=0}^{\infty }\frac{(n)_{j}(-k_{\alpha })^{j}x^{\nu j+\nu n-2}}{%
j!\Gamma (\nu j+\nu n-1)}  \label{ds}
\end{equation}%
and thus, by (\ref{ll}) and (\ref{ml2}), we get%
\begin{eqnarray}
\mathcal{L}\{D_{x}^{\alpha ,\nu }\,f_{\mathcal{S}_{\psi }}(x);\eta \} &=&%
\frac{\eta ^{\nu -1}e^{-t}}{\eta ^{\nu }+k_{\alpha }}\sum_{n=1}^{\infty }%
\frac{\left( k_{\alpha }t\right) ^{n}}{n!}\sum_{j=0}^{\infty }\frac{%
(n)_{j}(-k_{\alpha })^{j}}{j!\eta ^{\nu j+\nu n-1}}  \label{ll3} \\
&=&\frac{\eta ^{\nu }e^{-t}}{\eta ^{\nu }+k_{\alpha }}\sum_{n=1}^{\infty }%
\frac{\left( k_{\alpha }t/\eta ^{\nu }\right) ^{n}}{n!(n-1)!}%
\sum_{j=0}^{\infty }\frac{(n+j-1)!(-k_{\alpha }/\eta ^{\nu })^{j}}{j!}
\notag \\
&=&\frac{\eta ^{\nu }e^{-t}}{\eta ^{\nu }+k_{\alpha }}\left[ e^{\frac{%
k_{\alpha }}{\eta ^{\nu }+k_{\alpha }}t}-1\right] =\frac{\eta ^{\nu }}{\eta
^{\nu }+k_{\alpha }}\left[ e^{-\frac{\eta ^{\nu }}{\eta ^{\nu }+k_{\alpha }}%
t}-e^{-t}\right] .  \notag
\end{eqnarray}%
Therefore equation (\ref{abeq}) is proved to hold, by considering (\ref{ll3}%
) together with the time-derivative of (\ref{ml3}).
\end{proof}

\begin{remark}
The previous result holds only for $\nu <1$, since formula (\ref{ds}) can be
rewritten as
\begin{equation*}
\frac{\partial }{\partial x}f_{\mathcal{S}_{\psi
}}(x,t)=e^{-t}\sum_{n=1}^{\infty }\frac{\left( k_{\alpha }t\right) ^{n}}{n!}%
\sum_{j=1}^{\infty }\frac{(n)_{j}(-k_{\alpha })^{j}x^{\nu j+\nu n-2}}{%
j!\Gamma (\nu j+\nu n-1)}+e^{-t}\sum_{n=1}^{\infty }\frac{\left( k_{\alpha
}t\right) ^{n}}{n!}\frac{x^{\nu n-2}}{\Gamma (\nu n-1)}
\end{equation*}%
where the first term of the last sum (i.e. for $n=1$) vanishes when $\nu =1.$
Therefore, in this special case, we get%
\begin{eqnarray}
&&\mathcal{L}\{D_{x}^{\alpha ,1}\,f_{\mathcal{S}_{\psi }}(x);\eta \}
\label{ll4} \\
&=&\frac{\eta e^{-t}}{\eta +k_{\alpha }}\sum_{n=1}^{\infty }\frac{\left(
k_{\alpha }t/\eta \right) ^{n}}{n!(n-1)!}\sum_{j=1}^{\infty }\frac{%
(n+j-1)!(-k_{\alpha }/\eta )^{j}}{j!}+\frac{\eta e^{-t}}{\eta +k_{\alpha }}%
\sum_{n=1}^{\infty }\frac{\left( k_{\alpha }t/\eta \right) ^{n}}{n!}  \notag
\\
&=&\frac{\eta e^{-t}}{\eta +k_{\alpha }}\left[ e^{\frac{k_{\alpha }}{\eta
+k_{\alpha }}t}-1\right] -\frac{\eta e^{-t}}{\eta +k_{\alpha }}\left[ e^{%
\frac{k_{\alpha }}{\eta }t}-1\right] +\frac{\eta e^{-t}}{\eta +k_{\alpha }}%
\left[ e^{\frac{k_{\alpha }}{\eta }t}-1-\frac{k_{\alpha }t}{\eta }\right]
\notag \\
&=&\frac{\eta }{\eta +k_{\alpha }}\left[ e^{-\frac{\eta }{\eta +k_{\alpha }}%
t}-e^{-t}\right] -\frac{k_{\alpha }te^{-t}}{\eta +k_{\alpha }},  \notag
\end{eqnarray}%
which coincides with the Laplace transform of (\ref{uf2}), while differs
from (\ref{ll3})$,$ with $\nu =1$.
\end{remark}

\begin{remark}
As for the exponential case, also Theorem 3 can be, alternatively, proved
directly as follows: the time-derivative of (\ref{abden}) reads%
\begin{eqnarray}
&&\frac{\partial }{\partial t}f_{\mathcal{S}_{\psi }}(x,t)  \label{dt2} \\
&=&-\frac{e^{-t}}{x}\sum_{n=1}^{\infty }\frac{\left( k_{\alpha }tx^{\nu
}\right) ^{n}}{n!}\left[ E_{\nu ,\nu n}^{n}(-k_{\alpha }x^{\nu })-k_{\alpha
}x^{\nu }E_{\nu ,\nu (n+1)}^{n+1}(-k_{\alpha }x^{\nu })\right] +  \notag \\
&&+k_{\alpha }e^{-t}x^{\nu -1}E_{\nu ,\nu }(-k_{\alpha }x^{\nu })  \notag \\
&=&-\frac{e^{-t}}{x}\sum_{n=1}^{\infty }\frac{\left( k_{\alpha }tx^{\nu
}\right) ^{n}}{n!}E_{\nu ,\nu n}^{n+1}(-k_{\alpha }x^{\nu })+k_{\alpha
}e^{-t}x^{\nu -1}E_{\nu ,\nu }(-k_{\alpha }x^{\nu }),  \notag
\end{eqnarray}%
where, in the last step, we applied formula (3.6) in \cite{BEG2}, for $m=n+1$
and $z=0.$ By considering (\ref{ds}), together with (\ref{ab}), we get%
\begin{eqnarray}
&&D_{x}^{\alpha ,\nu }\,f_{\mathcal{S}_{\psi }}(x,t)  \label{ds2} \\
&=&e^{-t}\sum_{n=1}^{\infty }\frac{\left( k_{\alpha }t\right) ^{n}}{n!}%
\sum_{j=0}^{\infty }\frac{(n)_{j}(-k_{\alpha })^{j}x^{\nu j+\nu n-1}}{%
j!\Gamma (\nu j+\nu n-1)}\sum_{l=0}^{\infty }\frac{(-k_{\alpha })^{l}x^{\nu
l}}{\Gamma (\nu l+1)}\int_{0}^{1}(1-y)^{\nu j+\nu n-2}y^{\nu l}dy  \notag \\
&=&\frac{e^{-t}}{x}\sum_{n=1}^{\infty }\frac{\left( k_{\alpha }tx^{\nu
}\right) ^{n}}{n!}\sum_{j=0}^{\infty }\frac{(n)_{j}}{j!}\sum_{l=0}^{\infty }%
\frac{(-k_{\alpha }x^{\nu })^{l+j}}{\Gamma (\nu l+\nu j+\nu n)}  \notag \\
&=&\frac{e^{-t}}{x}\sum_{n=1}^{\infty }\frac{\left( k_{\alpha }tx^{\nu
}\right) ^{n}}{n!}\sum_{m=0}^{\infty }\frac{(-k_{\alpha }x^{\nu })^{m}}{%
\Gamma (\nu m+\nu n)}\sum_{j=0}^{m}\binom{n+j-1}{j}  \notag \\
&=&\frac{e^{-t}}{x}\sum_{n=1}^{\infty }\frac{\left( k_{\alpha }tx^{\nu
}\right) ^{n}}{n!}\sum_{m=0}^{\infty }\frac{(n+1)_{m}(-k_{\alpha }x^{\nu
})^{m}}{m!\Gamma (\nu m+\nu n)}=\frac{e^{-t}}{x}\sum_{n=1}^{\infty }\frac{%
\left( k_{\alpha }tx^{\nu }\right) ^{n}}{n!}E_{\nu ,\nu n}^{n+1}(-k_{\alpha
}x^{\nu }),  \notag
\end{eqnarray}%
where $(\gamma )_{j}:=\gamma (\gamma +1)...(\gamma +j-1),$ $j=0,1,...$ The
last step follows by the well-known formula $\sum_{j=0}^{m}\binom{s+j}{j}=%
\binom{s+m+1}{m};$ by considering (\ref{ds2}) with (\ref{dt2}), we obtain (%
\ref{abeq}).
\end{remark}

\subsection{Incomplete-gamma kernel}

Let $\Gamma (\rho ;x):=\int_{x}^{+\infty }e^{-w}w^{\rho -1}dw$ be the
upper-incomplete gamma function (which is defined for any $\rho ,x\in
\mathbb{R}$ and is real-valued for $x\geq 0$).

We now consider the operator defined in (\ref{ig}), for a differentiable
function $f$ in $AC_{loc}\left( 0,+\infty \right) .$ It is immediate to
check that, when $\rho =1$, it reduces to (\ref{cf}). Moreover, for $z=0$,
the kernel is equal to one and thus it is non-singular in the origin. As in
the Mittag-Leffler case, the following result holds only for $\rho <1.$

\begin{theorem}
Let $\mathcal{D}_{x}^{\alpha ,\rho }$ be the convolution operator defined in
(\ref{ig}). Then the solution to the following equation, for $\rho \in
(0,1), $%
\begin{equation}
\frac{\partial }{\partial t}f(x,t)=-\,\mathcal{D}_{x}^{\alpha ,\rho
}\,f(x,t)+\frac{k_{\alpha }^{\rho }x^{\rho -1}}{\Gamma (\rho )}%
e^{-t-k_{\alpha }x},\qquad x\geq 0,t\geq 0,\text{ }\alpha \in (0,1),
\label{ggeq}
\end{equation}%
with initial condition $f(x,0)=0,$ is given by
\begin{equation}
f_{\mathcal{S}_{\psi }}(x,t)=\frac{e^{-t-k_{\alpha }x}}{x}\sum_{n=1}^{\infty
}\frac{\left( k_{\alpha }^{\rho }tx^{\rho }\right) ^{n}}{n!\Gamma (\rho n)}%
,\qquad x,t\geq 0.  \label{gg4}
\end{equation}%
Moreover (\ref{gg4}) is the density of the absolutely continuous component
of $\mathcal{S}_{\psi }$ defined in (\ref{ss}), for $X_{j}^{\psi }$ i.i.d.
random variables with density function%
\begin{equation}
f_{X_{j}^{\psi }}(x)=\frac{k_{\alpha }^{\rho }}{\Gamma (\rho )}x^{\rho
-1}e^{-k_{\alpha }x},\qquad x\geq 0.  \label{gg5}
\end{equation}
\end{theorem}

\begin{proof}
We can write, analogously to the previous cases, that $\overline{\mu }%
(x)=\Gamma \left( \rho ;k_{\alpha }x\right) /\Gamma \left( \rho \right) $
and
\begin{eqnarray}
\,\psi (\eta ) &=&\eta \int\limits_{0}^{+\infty }e^{-\eta x}\frac{\Gamma
\left( \rho ;k_{\alpha }x\right) }{\Gamma \left( \rho \right) }dx=\frac{\eta
}{\Gamma \left( \rho \right) }\int\limits_{0}^{+\infty }e^{-\eta
x}\int\limits_{k_{\alpha }x}^{+\infty }e^{-w}w^{\rho -1}dwdx  \label{psi} \\
&=&\frac{\eta k_{\alpha }^{\rho }}{\Gamma \left( \rho \right) }%
\int\limits_{0}^{+\infty }e^{-k_{\alpha }y}y^{\rho
-1}\int\limits_{0}^{y}e^{-\eta x}dxdy=1-\frac{k_{\alpha }^{\rho }}{\left(
\eta +k_{\alpha }\right) ^{\rho }}.  \notag
\end{eqnarray}%
Therefore we have that%
\begin{equation}
\widetilde{f}_{\mathcal{S}_{\psi }}(\eta ,t)=e^{-t+\frac{k_{\alpha }^{\rho }%
}{\left( \eta +k_{\alpha }\right) ^{\rho }}t}-e^{-t}.  \label{loc}
\end{equation}%
By taking the inverse Laplace transform of (\ref{loc}), we get (\ref{gg4}).
The density in (\ref{gg5}) is correspondingly obtained, in view of (\ref{fr}%
), by inverting%
\begin{equation*}
1-\psi (\eta )=\frac{k_{\alpha }^{\rho }}{\left( \eta +k_{\alpha }\right)
^{\rho }}.
\end{equation*}%
As in the previous case, in order to prove that (\ref{gg4}) satisfies
equation (\ref{ggeq}), we can not apply (\ref{lt3}), since also the function
(\ref{gg4}) is infinite in the origin. From (\ref{gg4}), we have that
\begin{equation}
\frac{\partial }{\partial x}f_{\mathcal{S}_{\psi }}(x,t)=-\frac{k_{\alpha
}e^{-t-k_{\alpha }x}}{x}\sum_{n=1}^{\infty }\frac{\left( k_{\alpha }^{\rho
}tx^{\rho }\right) ^{n}}{n!\Gamma (\rho n)}+\frac{e^{-t-k_{\alpha }x}}{x^{2}}%
\sum_{n=1}^{\infty }\frac{\left( k_{\alpha }^{\rho }tx^{\rho }\right) ^{n}}{%
n!\Gamma (\rho n-1)}  \label{dx}
\end{equation}%
and%
\begin{equation*}
\mathcal{L}\{\frac{\partial }{\partial x}f_{\mathcal{S}_{\psi }}(x,t);\eta
\}=\eta \left[ e^{-\frac{\eta ^{\rho }}{\left( \eta +k_{\alpha }\right)
^{\rho }}t}-e^{-t}\right] .
\end{equation*}%
Then, by taking into account the analogue of (\ref{ll}) together with (\ref%
{psi}), we get%
\begin{eqnarray}
\mathcal{L}\{\mathcal{D}_{x}^{\alpha ,\rho }\,f(x,t);\eta \} &=&\frac{1}{%
\Gamma \left( \rho \right) }\mathcal{L}\{\frac{\partial }{\partial x}f_{%
\mathcal{S}_{\psi }}(x,t);\eta \}\mathcal{L}\{\Gamma \left( \rho ;k_{\alpha
}x\right) ;\eta \}  \label{uu} \\
&=&\frac{\left( \eta +k_{\alpha }\right) ^{\rho }-k_{\alpha }^{\rho }}{%
\left( \eta +k_{\alpha }\right) ^{\rho }}\left[ e^{-\frac{\left( \eta
+k_{\alpha }\right) ^{\rho }-k_{\alpha }^{\rho }}{\left( \eta +k_{\alpha
}\right) ^{\rho }}t}-e^{-t}\right] ,  \notag
\end{eqnarray}%
so that (\ref{ggeq}) easily follows, by considering the time-derivative of (%
\ref{loc}).
\end{proof}

\begin{remark}
The previous result holds only for $\rho $ strictly smaller than $1$, since
formula (\ref{dx}) can be rewritten as
\begin{equation*}
\frac{\partial }{\partial x}f_{\mathcal{S}_{\psi }}(x,t)=-\frac{k_{\alpha
}e^{-t-k_{\alpha }x}}{x}\sum_{n=1}^{\infty }\frac{\left( k_{\alpha }^{\rho
}tx^{\rho }\right) ^{n}}{n!\Gamma (\rho n)}+\frac{e^{-t-k_{\alpha }x}}{x^{2}}%
\sum_{n=2}^{\infty }\frac{\left( k_{\alpha }^{\rho }tx^{\rho }\right) ^{n}}{%
n!\Gamma (\rho n-1)}+\frac{k_{\alpha }^{\rho }tx^{\rho -2}e^{-t-k_{\alpha }x}%
}{\Gamma (\rho -1)}
\end{equation*}%
where the last term vanishes when $\rho =1.$ Therefore, in this special
case, we get%
\begin{eqnarray*}
&&\mathcal{L}\{\mathcal{D}_{x}^{\alpha ,1}\,f_{\mathcal{S}_{\psi
}}(x,t);\eta \} \\
&=&-\frac{k_{\alpha }e^{-t}}{\eta +k_{\alpha }}\sum_{n=1}^{\infty }\frac{%
\left( k_{\alpha }t\right) ^{n}}{n!(\eta +k_{\alpha })^{n}}%
+e^{-t}\sum_{n=2}^{\infty }\frac{\left( k_{\alpha }t\right) ^{n}}{n!(\eta
+k_{\alpha })^{n}} \\
&=&\frac{\eta }{\eta +k_{\alpha }}\left[ e^{-\frac{\eta t}{\eta +k_{\alpha }}%
}-e^{-t}\right] -\frac{k_{\alpha }te^{-t}}{\eta +k_{\alpha }}
\end{eqnarray*}%
which coincides with the Laplace transform of (\ref{uf2}), while differs
from (\ref{uu})$,$ with $\rho =1$.
\end{remark}

\section{The distributed case}

We now extend the results of the previous section by generalizing the
kernels to the case of random parameters and obtaining a distributed order
operator. We can provide the explicit solution of the corresponding
equations, at least under simplifying assumptions.

We start by considering the exponential kernel in (\ref{cf}) and by assuming
that $\alpha ,$ instead of being a fixed parameter, is a random variable,
taking values in $(0,1)$, with a given distribution.

\begin{definition}
Let $\alpha $ be a random variable, with values in $\left( 0,1\right) $
almost surely, and let $q:=q(\alpha )$ be its density function. Then, we
define the following convolution operator%
\begin{equation}
^{CF}D_{x}^{q(\alpha )}f(x):=\int\limits_{0}^{x}\frac{d}{dz}f(z)\left(
\int_{0}^{1}e^{-\frac{\alpha }{1-\alpha }(x-z)}q(\alpha )d\alpha \right)
dz,\qquad x>0,  \label{qq}
\end{equation}%
for a differentiable function $f,$ in $AC_{loc}\left( 0,+\infty \right) .$
\end{definition}

It is immediate to see that, in the special case where $q(z)=\delta
(z-\alpha )$ (with $\delta (\cdot )$ denoting the Dirac delta function),
formula (\ref{qq}) reduces to (\ref{cf}) (under the assumption that $%
B(\alpha )=1-\alpha $).

We will assume hereafter that $q$ is such that $\int_{0}^{1}\alpha q(\alpha
)/(1-\alpha )d\alpha <\infty .$ This assumption implies, for example, in the
Beta case, i.e. for $q(\alpha )=\alpha ^{r-1}(1-\alpha )^{s-1}/B(r,s),$ $%
\alpha \in (0,1),$ (where $B(r,s)$ is the Beta function) that we must choose
$s>1.$

\begin{theorem}
Let $^{CF}D_{x}^{q(\alpha )}$ be the convolution operator defined in (\ref%
{qq}); let moreover $A_{q}(x):=\int_{0}^{1}k_{\alpha }e^{-k_{\alpha
}x}q(\alpha )d\alpha $, for $x\geq 0$, $k_{\alpha }=\alpha /(1-\alpha )$ and
$B_{q}(x):=\left[ \int_{0}^{1}k_{\alpha }q(\alpha )d\alpha \right] \left[
\int_{0}^{1}e^{-k_{\alpha }x}q(\alpha )d\alpha \right] $, for $x\geq 0.$Then
the solution to the following equation%
\begin{equation}
\frac{\partial }{\partial t}f(x,t)=-\,^{CF}D_{x}^{q(\alpha
)}\,f(x,t)+e^{-t}A_{q}(x)-te^{-t}B_{q}(x),\qquad x\geq 0,t\geq 0,
\label{qq2}
\end{equation}%
(with initial condition $f(x,0)=0),$ is given by the density of the
absolutely continuous component of $\mathcal{S}_{\psi }$ defined in (\ref{ss}%
), for $X_{j}^{\psi }$ independent and identically distributed r.v.'s, $%
j=1,2...$., with density $f_{X_{j}^{\psi }}(x)=A_{q}(x),$ $x\geq 0.$
\end{theorem}

\begin{proof}
In this case we have that $\overline{\mu }(x)=\int_{0}^{1}e^{-k_{\alpha
}x}q(\alpha )d\alpha $; thus we get%
\begin{equation*}
\psi (\eta )=\eta \int\limits_{0}^{+\infty }e^{-\eta
x}\int_{0}^{1}e^{-k_{\alpha }x}q(\alpha )d\alpha dx=\int_{0}^{1}\frac{\eta }{%
k_{\alpha }+\eta }q(\alpha )d\alpha
\end{equation*}%
and
\begin{equation}
\widetilde{f}_{\mathcal{S}_{\psi }}(\eta ,t)=e^{-t\int_{0}^{1}\frac{\eta }{%
k_{\alpha }+\eta }q(\alpha )d\alpha }-e^{-t}.  \label{qq3}
\end{equation}%
The density $f_{X_{j}^{\psi }}$ is obtained, in view of (\ref{fr}), by the
inversion of $1-\,\psi (\eta )=\int_{0}^{1}\frac{k_{\alpha }}{k_{\alpha
}+\eta }q(\alpha )d\alpha ,$ which gives $f_{X_{j}^{\psi
}}(x)=\int_{0}^{1}k_{\alpha }e^{-k_{\alpha }x}q(\alpha )d\alpha .$ Even if,
without any specific assumption on $q(\alpha )$, we can not write an
explicit form for the density $f_{\mathcal{S}_{\psi }}$ (since (\ref{qq3})
can not be inverted), we can nevertheless obtain its value in zero. Indeed,
by (\ref{lp}) and by considering that $f_{X^{\psi }}^{\ast n}(0)=0$, for $%
n>1,$ while $f_{X^{\psi }}^{\ast n}(0)=\int_{0}^{1}k_{\alpha }q(\alpha
)d\alpha $, for $n=1$, we get%
\begin{equation}
f_{\mathcal{S}_{\psi }}(0,t)=te^{-t}\int_{0}^{1}k_{\alpha }q(\alpha )d\alpha
,\qquad t\geq 0.  \label{qqo}
\end{equation}%
By definition (\ref{qq}) we have that%
\begin{eqnarray*}
&&\mathcal{L}\{^{CF}D_{x}^{q(\alpha )}\,f(x);\eta \}=\mathcal{L}%
\{\int_{0}^{1}e^{-k_{\alpha }x}q(\alpha )d\alpha ;\eta \}\mathcal{L}\{\frac{d%
}{dx}f(x);\eta \} \\
&=&\int_{0}^{1}\frac{q(\alpha )}{k_{\alpha }+\eta }d\alpha \left[ \eta
\widetilde{f}(\eta )-f(0)\right] .
\end{eqnarray*}%
In view of (\ref{qq3}) and (\ref{qqo}), we can write%
\begin{eqnarray}
&&\mathcal{L}\{^{CF}D_{x}^{q(\alpha )}\,f(x,t);\eta \}  \label{qqd} \\
&=&\int_{0}^{1}\frac{\eta q(\alpha )}{k_{\alpha }+\eta }d\alpha \left[
e^{-t\int_{0}^{1}\frac{\eta }{k_{\alpha }+\eta }q(\alpha )d\alpha }-e^{-t}%
\right] -te^{-t}\left[ \int_{0}^{1}\frac{q(\alpha )}{k_{\alpha }+\eta }%
d\alpha \right] \left[ \int_{0}^{1}k_{\alpha }q(\alpha )d\alpha \right] .
\notag
\end{eqnarray}%
By comparing (\ref{qqd}) with the time-derivative of (\ref{qq3}), we easily
prove that (\ref{qq2}) holds.
\end{proof}

\

For $q(z)=\delta (z-\alpha )$, we obtain that $A_{q}(x)=B_{q}(x)=k_{\alpha
}e^{-k_{\alpha }x}$ and the previous result coincides with that presented in
Theorem 1. In another special case, i.e. for $q(z)=q_{1}\delta (z-\alpha
_{1})+q_{2}\delta (z-\alpha _{2})$, for $0<\alpha _{1}<\alpha _{2}<1$ and $%
q_{1},q_{2}\in \lbrack 0,1]$ such that $q_{1}+q_{2}=1,$ we can obtain an
explicit form of the density $f_{\mathcal{S}_{\psi }},$ which generalizes (%
\ref{den}).

\begin{corollary}
Let $^{CF}D_{x}^{\alpha _{1},\alpha _{2}}$ be the convolution operator
defined in (\ref{qq}) under the assumption that $q(z)=q_{1}\delta (z-\alpha
_{1})+q_{2}\delta (z-\alpha _{2})$, $z\geq 0,$ for $0<\alpha _{1}<\alpha
_{2}<2$ and $q_{1},q_{2}\in \lbrack 0,1]$ such that $q_{1}+q_{2}=1$ $.$Then
the solution to the following equation%
\begin{eqnarray}
\frac{\partial }{\partial t}f(x,t) &=&-\,^{CF}D_{x}^{\alpha _{1},\alpha
_{2}}\,f(x,t)+[q_{1}k_{\alpha _{1}}e^{-t-k_{\alpha _{1}}x}+q_{2}k_{\alpha
_{2}}e^{-t-k_{\alpha _{2}}x}]+  \label{qqeq} \\
&&-t(q_{1}k_{\alpha _{1}}+q_{2}k_{\alpha _{2}})[q_{1}e^{-t-k_{\alpha
_{1}}x}+q_{2}e^{-t-k_{\alpha _{2}}x}],  \notag
\end{eqnarray}%
$x\geq 0,t\geq 0,$ (with initial condition $f(x,0)=0),$ is given by%
\begin{equation}
f_{\mathcal{S}_{\psi }}(x,t)=\frac{e^{-t-k_{\alpha _{2}}x}}{x}%
\sum_{n=1}^{\infty }\frac{(q_{2}k_{\alpha _{2}}xt)^{n}}{n!}\sum_{j=0}^{n}%
\binom{n}{j}\left( \frac{q_{1}k_{\alpha _{1}}}{q_{2}k_{\alpha _{2}}}\right)
^{j}E_{1,n}^{j}((k_{\alpha _{2}}-k_{\alpha _{1}})x),\qquad x\geq 0.
\label{qqr}
\end{equation}%
Moreover (\ref{qqr}) is the density of the absolutely continuous component
of $\mathcal{S}_{\psi }$ defined in (\ref{ss}), for $X_{j}^{\psi }$
independent and identically distributed r.v.'s, $j=1,2...$., with density $%
f_{X_{j}^{\psi }}(x)=q_{1}k_{\alpha _{1}}e^{-k_{\alpha
_{1}}x}+q_{2}k_{\alpha _{2}}e^{-k_{\alpha _{2}}x},$ $x\geq 0.$
\end{corollary}

\begin{proof}
Equation (\ref{qqeq}) is obtained as special case of (\ref{qq2}). We only
need to prove that (\ref{qqr}) satisfies equation (\ref{qqeq}), by checking
that its Laplace transform coincides with (\ref{qq3}), as follows%
\begin{eqnarray*}
\widetilde{f}_{\mathcal{S}_{\psi }}(\eta ,t) &=&e^{-t}\sum_{n=1}^{\infty }%
\frac{(q_{2}k_{\alpha _{2}}t)^{n}}{n!}\sum_{j=0}^{n}\binom{n}{j}\left( \frac{%
q_{1}k_{\alpha _{1}}}{q_{2}k_{\alpha _{2}}}\right) ^{j}\int_{0}^{+\infty
}e^{-(k_{\alpha _{2}}+\eta )x}x^{n-1}E_{1,n}^{j}((k_{\alpha _{2}}-k_{\alpha
_{1}})x)dx \\
&=&e^{-t}\sum_{n=1}^{\infty }\frac{(q_{2}k_{\alpha _{2}}t)^{n}}{n!}%
\sum_{j=0}^{n}\binom{n}{j}\left( \frac{q_{1}k_{\alpha _{1}}}{q_{2}k_{\alpha
_{2}}}\right) ^{j}\frac{(k_{\alpha _{2}}+\eta )^{j-n}}{(k_{\alpha _{1}}+\eta
)^{j}} \\
&=&e^{-t}\sum_{n=1}^{\infty }\frac{(q_{2}k_{\alpha _{2}}t/(k_{\alpha
_{1}}+\eta ))^{n}}{n!}\left( \frac{q_{1}k_{\alpha _{1}}}{q_{2}k_{\alpha _{2}}%
}+\frac{k_{\alpha _{1}}+\eta }{k_{\alpha _{2}}+\eta }\right) ^{n} \\
&=&e^{-t}\left[ \exp \left\{ \frac{k_{\alpha _{1}}k_{\alpha _{2}}+\eta
(q_{1}k_{\alpha _{1}}+q_{2}k_{\alpha _{2}})}{(k_{\alpha _{1}}+\eta
)(k_{\alpha _{2}}+\eta )}t\right\} -1\right] .
\end{eqnarray*}%
The previous expression coincides with (\ref{qq3}), by the assumption on $q$
and by considering that $q_{1}+q_{2}=1$.

\
\end{proof}

In the case of the Mittag-Leffler kernel, we generalize the operator (\ref%
{ab}) to the distributed case, as follows.

\begin{definition}
Let $\nu $ be a random variable, with values in $\left( 0,1\right) $ almost
surely, and let $q:=q(\nu ),$ be its density function. Then, we define the
following convolution operator
\end{definition}

\begin{equation}
D_{x}^{\alpha ,q(\nu )}f(x):=\int\limits_{0}^{x}\frac{d}{dz}f(z)\left[
\int_{0}^{1}E_{\nu }\left( -k_{\alpha }(x-z)^{\nu }\right) q(\nu )d\nu %
\right] dz,\qquad x>0,\text{ }\alpha \in (0,1),  \label{nu}
\end{equation}%
for a differentiable function $f$ in $AC_{loc}\left( 0,+\infty \right) .$

\begin{theorem}
Let $D_{x}^{\alpha ,q(\nu )}$ be the convolution operator defined in (\ref%
{nu}); then the solution to the following equation%
\begin{equation}
\frac{\partial }{\partial t}f(x,t)=-\,D_{x}^{\alpha ,q(\nu
)}\,f(x,t)+k_{\alpha }e^{-t}\int_{0}^{1}x^{\nu -1}E_{\nu ,\nu }(-k_{\alpha
}x^{\nu })q(\nu )d\nu ,\qquad x\geq 0,t\geq 0,  \label{eq3}
\end{equation}%
(with initial condition $f(x,0)=0),$ is given by%
\begin{equation}
f_{\mathcal{S}_{\psi }}(x,t)=\frac{e^{-t}}{x}\sum_{n=1}^{\infty }\frac{%
\left( k_{\alpha }t\right) ^{n}}{n!}\int_{0}^{1}x^{\nu }E_{\nu ,\nu
n}^{n}(-k_{\alpha }x^{\nu })q(\nu )d\nu .  \label{den3}
\end{equation}%
Moreover, (\ref{den3}) is the density of the absolutely continuous component
of $\mathcal{S}_{\psi }$ defined in (\ref{ss}), for $X_{j}^{\psi }$
independent and identically distributed r.v.'s, $j=1,2...$., with density $%
f_{X_{j}^{\psi }}(x)=k_{\alpha }\int_{0}^{1}x^{\nu -1}E_{\nu ,\nu }\left(
-k_{\alpha }x^{\nu }\right) q(\nu )d\nu ,$ $x\geq 0.$
\end{theorem}

\begin{proof}
Following the same lines of the non-distributed case, we can write that%
\begin{equation}
\widetilde{f}_{\mathcal{S}_{\psi }}(\eta ,t)=e^{-t\int_{0}^{1}\frac{\eta
^{\nu }}{\eta ^{\nu }+k_{\alpha }}q(\nu )d\nu }-e^{-t}  \notag
\end{equation}%
and
\begin{equation*}
\mathcal{L}\{\mathcal{D}_{x}^{\alpha ,q(\nu )}\,f_{\mathcal{S}_{\psi
}}(x,t);\eta \}=\int_{0}^{1}\frac{\eta ^{\nu -1}}{\eta ^{\nu }+k_{\alpha }}%
q(\nu )d\nu \left[ e^{-t\int_{0}^{1}\frac{\eta ^{\nu }}{\eta ^{\nu
}+k_{\alpha }}q(\nu )d\nu }-e^{-t}\right] ,
\end{equation*}%
so that equation (\ref{eq3}) easily follows.
\end{proof}

\

As far as the incomplete-gamma kernel case is concerned, we can extend the
results of section 2.3 by considering the operator defined in the following

\begin{definition}
Let $\rho $ be a random variable, with values in $\left( 0,1\right) $ almost
surely, and let $q:=q(\rho ),$ be its density function. Then, we define the
following convolution operator%
\begin{equation}
\mathcal{D}_{x}^{\alpha ,q(\rho )}f(x):=\int\limits_{0}^{x}\frac{d}{dz}%
f(z)\left( \int_{0}^{1}\frac{\Gamma \left( \rho ;k_{\alpha }z\right) }{%
\Gamma (\rho )}q(\rho )d\rho \right) dz,\qquad x>0,\text{ }\alpha \in (0,1),
\label{qr}
\end{equation}%
for a differentiable function $f$ in $AC_{loc}\left( 0,+\infty \right) .$
\end{definition}

Then, in this case, we prove the following

\begin{theorem}
Let $\mathcal{D}_{x}^{\alpha ,q(\rho )}$ be the convolution operator (\ref%
{qr}); then the solution to the following equation%
\begin{equation}
\frac{\partial }{\partial t}f(x,t)=-\,\mathcal{D}_{x}^{\alpha ,q(\rho
)}\,f(x,t)+e^{-t-k_{\alpha }x}\int_{0}^{1}\frac{k_{\alpha }^{\rho }x^{\rho
-1}}{\Gamma (\rho )}q(\rho )d\rho ,\qquad x\geq 0,t\geq 0,
\end{equation}%
(with initial condition $f(x,0)=0),$ is given by the density of the
absolutely continuous component of $\mathcal{S}_{\psi }$ defined in (\ref{ss}%
), for $X_{j}^{\psi }$ independent and identically distributed r.v.'s, $%
j=1,2...$., with density $f_{X_{j}^{\psi }}(x)=e^{-k_{\alpha }x}\int_{0}^{1}%
\frac{k_{\alpha }^{\rho }x^{\rho -1}}{\Gamma (\rho )}q(\rho )d\rho ,$ $x\geq
0.$
\end{theorem}

\begin{proof}
Analogously to the non-distributed case, we can write that%
\begin{equation}
\widetilde{f}_{\mathcal{S}_{\psi }}(\eta ,t)=e^{-t\int_{0}^{1}\frac{\left(
\eta +k_{\alpha }\right) ^{\rho }-k_{\alpha }^{\rho }}{\left( \eta
+k_{\alpha }\right) ^{\rho }}q(\rho )d\rho }-e^{-t}  \notag
\end{equation}%
and
\begin{equation*}
\mathcal{L}\{\mathcal{D}_{x}^{\alpha ,\rho }\,f_{\mathcal{S}_{\psi
}}(x,t);\eta \}=\int_{0}^{1}\frac{\left( \eta +k_{\alpha }\right) ^{\rho
}-k_{\alpha }^{\rho }}{\left( \eta +k_{\alpha }\right) ^{\rho }}q(\rho
)d\rho \left[ e^{-t\int_{0}^{1}\frac{\left( \eta +k_{\alpha }\right) ^{\rho
}-k_{\alpha }^{\rho }}{\left( \eta +k_{\alpha }\right) ^{\rho }}q(\rho
)d\rho }-e^{-t}\right] .
\end{equation*}%
On the other hand,
\begin{equation*}
\frac{\partial }{\partial t}\widetilde{f}_{\mathcal{S}_{\psi }}(\eta
,t)=-e^{-t\int_{0}^{1}\frac{\left( \eta +k_{\alpha }\right) ^{\rho
}-k_{\alpha }^{\rho }}{\left( \eta +k_{\alpha }\right) ^{\rho }}q(\rho
)d\rho }\int_{0}^{1}\frac{\left( \eta +k_{\alpha }\right) ^{\rho }-k_{\alpha
}^{\rho }}{\left( \eta +k_{\alpha }\right) ^{\rho }}q(\rho )d\rho +e^{-t}.
\end{equation*}
\end{proof}

As for the exponential kernel case, when $q(z)=q_{1}\delta (z-\rho
_{1})+q_{2}\delta (z-\rho _{2})$, for $0<\rho _{1}<\rho _{2}<1,$ we can
obtain an explicit form of the density $f_{\mathcal{S}_{\psi }},$ which
generalizes (\ref{gg4}).

\begin{corollary}
Let $\mathcal{D}_{x}^{\alpha ,q(\rho )}$ be the convolution operator defined
in (\ref{qr}) under the assumption that $q(z)=q_{1}\delta (z-\rho
_{1})+q_{2}\delta (z-\rho _{2})$, for $0<\rho _{1}<\rho _{2}<1$ and $%
q_{1},q_{2}\in \lbrack 0,1]$ such that $q_{1}+q_{2}=1$. Then the solution to
the following equation%
\begin{equation}
\frac{\partial }{\partial t}f(x,t)=-\,\mathcal{D}_{x}^{\alpha ,q(\rho
)}\,f(x,t)+e^{-t-k_{\alpha }x}\left[ \frac{q_{1}k_{\alpha }^{\rho
_{1}}x^{\rho _{1}-1}}{\Gamma (\rho _{1})}+\frac{q_{2}k_{\alpha }^{\rho
_{2}}x^{\rho _{2}-1}}{\Gamma (\rho _{2})}\right] ,  \notag
\end{equation}%
$x\geq 0,t\geq 0,$ (with initial condition $f(x,0)=0),$ is given by%
\begin{equation}
f_{\mathcal{S}_{\psi }}(x,t)=\frac{e^{-t-k_{\alpha }x}}{x}\sum_{l=0}^{\infty
}\frac{(q_{1}k_{\alpha }^{\rho _{1}}x^{\rho _{1}}t)^{l}}{l!}W_{\rho
_{2},\rho _{1}l}(q_{2}k_{\alpha }^{\rho _{2}}x^{\rho _{2}}t),\qquad x\geq 0.
\label{ggr}
\end{equation}%
Moreover (\ref{ggr}) is the density of the absolutely continuous component
of $\mathcal{S}_{\psi }$ defined in (\ref{ss}), for $X_{j}^{\psi }$
independent and identically distributed r.v.'s, $j=1,2...$., with density $%
f_{X_{j}^{\psi }}(x)=\frac{q_{1}k_{\alpha }^{\rho _{1}}x^{\rho _{1}-1}}{%
\Gamma (\rho _{1})}+\frac{q_{2}k_{\alpha }^{\rho _{2}}x^{\rho _{2}-1}}{%
\Gamma (\rho _{2})},$ $x\geq 0.$
\end{corollary}

\begin{proof}
We omit the details of the calculations which retrace those of Corollary 10.
\end{proof}

\section{Risk-theory applications and concluding remarks}

So far we have described the interplay between integro-differential
equations based on Caputo-like operators (with non-singular kernels) and the
densities of the corresponding stochastic processes. We proved that, ranging
from exponential to Mittag-Leffler or incomplete-gamma kernels, we obtain
compound Poisson processes with very different jump distributions.

In particular, passing from the Caputo-Fabrizio operator to the
Atangana-Baleanu one, we even loose the finiteness of all the moments of the
jumps $X_{j}^{\psi },$ $j=1,2,....$ Indeed, in the exponential case, $%
\mathbb{E}\mathcal{S}_{\psi }=\mathbb{E}X_{j}^{\psi }=1/k_{\alpha
}=(1-\alpha )/\alpha ,$ as can be easily checked by applying the Wald
formula and considering that the Poisson rate is unitary. Analogously, in
the distributed-order exponential case (i.e. with the operator defined in
Def.8), the mean value of $\mathcal{S}_{\psi }$ is equal to $\mathbb{E}%
\left( (1-\alpha )/\alpha \right) $.\

On the other hand, due to the well-known power-law behavior of the
Mittag-Leffler distribution given in (\ref{xx}), the expected value of $%
\mathcal{S}_{\psi }$ and $X_{j}^{\psi }$ are both infinite and the same
holds for the distributed order counterpart, obtained under Def.11.

Finally, the incomplete-gamma case can be considered intermediate between
the previous ones: indeed, the mean value is $\mathbb{E}\mathcal{S}_{\psi }=%
\mathbb{E}X_{j}^{\psi }=\rho /k_{\alpha }$ (or $\mathbb{E}(\rho /k_{\alpha
}) $ in the distributed case) and thus differs from the exponential case
only by the constant $\rho .$ Nevertheless the behavior of the density (\ref%
{gg5}) is completely different, in the origin, from the exponential one,
since it tends to infinity, as in the Mittag-Leffler case.

All the previous results can be applied to a continuous-time risk model. If
we define the risk reserve process%
\begin{equation}
R(t):=a+\beta t-\sum_{j=1}^{N(t)}X_{j}^{\psi },  \label{rt}
\end{equation}%
where $a\geq 0$ is the initial risk reserve, $\beta >0$ is the premium
collection rate and $X_{j}^{\psi },$ for $j=1,2,...,$ are the claims
occurring according to the Poisson process, then we have that $R(t)=a+\beta
t-S_{\psi }(t).$ If we denote the expected claim size by $\mu $ (i.e. $\mu :=%
\mathbb{E}X_{j}^{\psi }$), it is evident that, in order to fulfill the net
profit condition $\beta >\mu $ (considering that our Poisson rate is equal
to $1$), we must discard the case of Mittag-Leffler distributed claim sizes,
since the expected value of (\ref{xx}) is infinite.

In the other cases, by considering (\ref{rt}) together with (\ref{lap3}), we
can obtain the differential equation satisfied by the moment generating
function of $R(t),t\geq 0,$ when the claim size has distribution with
Laplace transform given in (\ref{fr}). By taking the first derivative (with
respect to the time variable) of%
\begin{eqnarray}
\Phi _{R}(\eta ,t) &:&=\mathbb{E}e^{\eta R(t)}=e^{\eta (a+\beta t)}\mathbb{E}%
e^{-\eta \sum_{j=1}^{N(t)}X_{j}^{\psi }}  \label{rt3} \\
&=&e^{\eta (a+\beta t)}\widetilde{f}_{S_{\psi }}(\eta ,t),  \notag
\end{eqnarray}%
we get%
\begin{eqnarray}
\frac{\partial }{\partial t}\Phi _{R}(\eta ,t) &=&\beta \eta \Phi _{R}(\eta
,t)+e^{\eta (a+\beta t)}\frac{\partial }{\partial t}\widetilde{f}_{S_{\psi
}}(\eta ,t)  \label{rt2} \\
&=&\beta \eta \Phi _{R}(\eta ,t)+e^{\eta (a+\beta t)}\left[ e^{-t}-\psi
(\eta )e^{-\psi (\eta )t}\right]  \notag \\
&=&[\beta \eta -\psi (\eta )]\Phi _{R}(\eta ,t)+e^{\eta (a+\beta t)-t}\left[
1-\psi (\eta )\right] ,  \notag
\end{eqnarray}%
which is satisfied by (\ref{rt3}). It is evident from the first line of (\ref%
{rt2}) that also the differential equation governing the risk reserve
process can be expressed in terms of the convolution operator $\mathcal{D}%
_{x}^{\psi }$ treated here.

\ \

\end{document}